\documentclass{amsart}
\usepackage[utf8]{inputenc}
\usepackage{amsmath}
\usepackage{amssymb}
\usepackage{amsthm}
\usepackage{array}
\usepackage{booktabs}
\usepackage{siunitx}
\usepackage[USenglish]{babel}
\usepackage[pdftex]{graphicx}
\usepackage{hyperref}

\theoremstyle{plain}
\newtheorem{theorem}{Theorem}[section]
\newtheorem*{theorem*}{Theorem}

\newtheorem*{corollary*}{Corollary}
\newtheorem{lemma}[theorem]{Lemma}
\newtheorem{proposition}[theorem]{Proposition}
\newtheorem{corollary}[theorem]{Corollary}

\theoremstyle{definition}

\newtheorem{remark}[theorem]{Remark}

\numberwithin{equation}{section}
\numberwithin{table}{section}

\def\Z{\mathbb{Z}}

\def\Q{\mathbb{Q}}

\def\R{\mathbb{R}}
\def\OK{\mathcal{O}_K}

\makeatletter
\def\pmod#1{\allowbreak\mkern10mu({\operator@font mod}\,\,#1)}

\makeatother

\newcolumntype{M}[1]{>{\centering $}p{#1}<{$}}
\newcolumntype{R}{>{\begin{math}}r<{\end{math}}}
\newcolumntype{L}{>{\begin{math}}l<{\end{math}}}
\newcolumntype{Z}{>{\begin{math}\mathbf\{}l<{\}\end{math}}}

\title{The Euclidean algorithm in quintic and septic cyclic fields}
\author{Pierre Lezowski}
\address{Universit\'e Blaise Pascal, Laboratoire de math\'ematiques UMR 6620,
Campus Universitaire des C\'ezeaux, BP 80026, 63171 Aubi\`ere C\'edex, France}
\email{pierre.lezowski@math.univ-bpclermont.fr}
\author{Kevin J. McGown}
\address{California State University, Chico, Department of Mathematics and Statistics, 601 E. Main St., Chico, CA, 95929, USA}
\email{kmcgown@csuchico.edu}
\date{}

\begin{document}

\maketitle




\section{Introduction}\label{S:intro}
Let $K$ be a number field with ring of integers $\OK$, and denote by $N=N_{K/\Q}$ the
absolute norm map.  For brevity, we will sometimes use the term field to mean a number field.
We call a number field $K$ norm-Euclidean if for every $\alpha,\beta\in\OK$, $\beta\neq 0$, there exists
$\gamma\in\OK$ such that $|N(\alpha-\gamma\beta)|<|N(\beta)|$.
Or equivalently, we may ask that for every $\xi\in K$ there exists $\gamma\in\OK$ such that
$|N(\xi-\gamma)|<1$.
In the quadratic setting, it is known that there are only finitely many norm-Euclidean fields and they have been identified
\cite{chatland.davenport, barnes.sd};
namely, a number field of the form $K=\Q(\sqrt{d})$ with $d$ squarefree is norm-Euclidean if and only if
\[
  d=-1,-2,-3,-7,-11,2, 3, 5, 6, 7, 11, 13, 17, 19, 21, 29, 33, 37, 41, 57, 73
  \,.
\]
This result was partially generalized by Heilbronn \cite{heilbronn:cyclic} as follows.
\begin{theorem}[Heilbronn]
Let $\ell$ be a prime. Then there are at most finitely many cyclic number fields
with degree $\ell$ which are norm-Euclidean.
\end{theorem}

However, Heilbronn provided no upper bound on the discriminant of such fields.
Building on work of Heilbronn, Godwin, and Smith
(see~\cite{heilbronn:cubic, godwin:1967, smith:1969, godwin.smith:1993}),
the second author proved the following result.
\begin{theorem}[{\cite[Theorem~1.1]{mcgown:jtnb}}]\label{T:cubic.GRH}
Assuming the GRH,
the norm-Euclidean cyclic cubic fields
are exactly those with discriminant
\[
    \Delta=7^2,9^2,13^2,19^2,31^2,37^2,43^2,61^2,67^2,103^2,109^2,127^2,157^2
    \,.
\]
\end{theorem}

Many of the results in the aforementioned paper go through for any cyclic field of odd prime degree.
The main goal of this
paper is to establish the analogue of Theorem~\ref{T:cubic.GRH} for
quintic and septic fields.
\begin{theorem}\label{T:quintic.septic}
Assume the GRH.
A cyclic field of degree $5$ is norm-Euclidean if and only if
$\Delta=11^4,31^4,41^4$.
A cyclic field of degree $7$ is norm-Euclidean if and only if
$\Delta=29^6,43^6$.
\end{theorem}
Although we cannot give a complete determination for degree $11$,
it appears
that for some degrees there are no norm-Euclidean cyclic fields whatsoever.
This was observed (but not proved) for degree $19$ in~\cite{mcgown:ijnt}.
We prove the following:
\begin{theorem}\label{T:19}
Assuming the GRH, there are no norm-Euclidean cyclic fields of degrees $19$,
$31$, $37$, $43$, $47$, $59$, $67$, $71$, $73$, $79$, $97$.
\end{theorem}

This list of primes is in no way intended to be complete, and
in fact, there may well be infinitely many primes $\ell$ for which 
there are no norm-Euclidean cyclic fields of degree $\ell$.

In other small degrees where we cannot give a complete determination, even under
the GRH,
we come very close.
Let $\mathcal{F}$ denote the collection of cyclic number fields of prime degree $3\leq \ell\leq 100$ and conductor $f$.
In this setting the
conductor-discriminant formula
tells us that
discriminant equals $\Delta=f^{\ell-1}$.
\begin{theorem}\label{T:list}
Assuming the GRH,
Table~\ref{Table:C}
contains all norm-Euclidean fields in $\mathcal{F}$.
(However, the possibility remains that some of these fields may not be norm-Euclidean.)
Moreover, even without the GRH, the table is complete for $f\leq \num{e13}$.
\begin{table}
\centering
\begin{tabular}{ll}
  \toprule
  $\ell$ & $f$ \\
  \midrule
  $3$ & $7$, $9$, $13$, $19$, $31$, $37$, $43$, $61$, $67$, $103$, $109$, $127$, $157$\\ 
  $5$ & $11$, $31$, $41$ \\ 
  $7$ & $29$, $43$\\
  $11$ & $23$, $67$, $331$\\
  $13$ & $53$, $131$\\
  $17$ & $137$\\
  $23$ & $47$, $139$ \\
  $29$ & $59$ \\
  $41$ & $83$\\
  $53$ & $107$\\
  $61$ & $367$\\
  $83$ & $167$, $499$\\
  $89$ & $179$\\
  \bottomrule
\end{tabular}
\caption{Possible norm-Euclidean fields in $\mathcal{F}$}\label{Table:C}
\end{table}
\end{theorem}

We remark that the top portion of this table (when $3\leq\ell\leq 30$) appeared in~\cite{mcgown:ijnt}
although it was unknown at the time (even under the GRH) whether the table
was complete.
A large part of establishing Theorems~\ref{T:quintic.septic}, ~\ref{T:19}, and~\ref{T:list}
was a computation that took 3.862 (one-core) years of CPU time
on a 96-core computer cluster.\footnote{The cluster consists of 8 compute nodes, each with twelve 2 GHz cores, and 64 GB memory per node.}

Finally, we also give a slight improvement on what is known unconditionally in the cubic case.
In~\cite{mcgown:ijnt}, it was shown that the conductor of any norm-Euclidean cyclic cubic field
not listed in Theorem~\ref{T:cubic.GRH} must lie in the interval
$(\num{e10}, \num{e70})$.  We improve this slightly, thereby obtaining:

\begin{theorem}\label{T:cubic.new}
Any norm-Euclidean cyclic cubic field not listed in Theorem~\ref{T:cubic.GRH} must have discriminant $\Delta=f^2$ with $f\equiv 1\pmod{3}$
where $f$ is a prime in the interval $(\num{2e14}, \num{e50})$.
\end{theorem}
Computing up to the new lower bound of $\num{2e14}$ required an additional $3.104$ years of CPU time on the same cluster.


\section{Summary}


For norm-Euclidean fields in $\mathcal{F}$ one has an upper bound on the conductor,
which is greatly improved with the use of the GRH;
in~\cite{mcgown:ijnt} and~\cite{mcgown:jtnb} the conductor bounds of
Table~\ref{Table:D} were established.
\begin{table}
\centering
\begin{tabular}{llllllll}
\toprule
$\ell$ & $3$ & $5$ & $7$ & $11$ & $13$ & $17$ & $19$ \\
\midrule
unconditional & $\num{e70}$ & $\num{e78}$ & $\num{e82}$ & $\num{e88}$ & $\num{e89}$ & $\num{e92}$ & $\num{e94}$\\
with the GRH & $\num{e11}$ & $\num{e12}$ & $\num{e13}$ & $\num{e13}$ & $\num{e14}$ & $\num{e14}$ & $\num{e14}$\\
\bottomrule
\end{tabular}
\caption{Conductor bounds for norm-Euclidean fields in $\mathcal{F}$
established in~\cite{mcgown:ijnt} and~\cite{mcgown:jtnb}}\label{Table:D}
\end{table}


In this paper, we establish the following improved bounds:
\begin{proposition}\label{P:0}
Assuming the GRH, Table~\ref{Table:A} gives
conductor bounds for norm-Euclidean fields in $\mathcal{F}$.
\end{proposition}

\begin{table}
\begin{tabular}{lllllllll}
\toprule
$\ell$ & $3$ & $5$ & $7$ & $11$ & $13$ & $17$ & $19$ & $23$\\
\cmidrule{2-9}
$f<$ & $\num{4e10}$ & $\num{6e10}$ & $\num{4e10}$ & $\num{2e11}$ & $\num{3e11}$ & $\num{6e11}$ & $\num{8e11}$ & $\num{2e11}$ \\
\midrule
$\ell$ & $29$ & $31$ & $37$ & $41$ & $43$ & $47$ & $53$ & $59$ \\
\cmidrule{2-9}
$f<$ & $\num{3e12}$ & $\num{3e12}$ & $\num{5e12}$ & $\num{6e12}$ & $\num{7e12}$ & $\num{9e12}$ & $\num{2e13}$ & $\num{2e13}$  \\
\midrule
$\ell$ & $61$ & $67$ & $71$ & $73$ & $79$ & $83$ & $89$ & $97$ \\
\cmidrule{2-9}
$f<$ &  $\num{2e13}$ & $\num{3e13}$ & $\num{3e13}$ & $\num{3e13}$ & $\num{4e13}$ & $\num{4e13}$ & $\num{5e13}$ & $\num{6e13}$ \\
\bottomrule
\end{tabular}
\caption{Conductor bounds for norm-Euclidean fields in $\mathcal{F}$ assuming the GRH}\label{Table:A}
\end{table}

In~\cite{mcgown:ijnt} computations were carried out that show the portion of Table~\ref{Table:C}
where $3\leq\ell\leq 30$ is complete up to $f=\num{e10}$.
We have extended these computations, thereby obtaining
the following unconditional result:
\begin{proposition}\label{P:1}
Table~\ref{Table:C} contains all possible norm-Euclidean fields in~$\mathcal{F}$ with $f\leq \num{e13}$.
Additionally, when $50\leq\ell\leq 100$, the table is complete up to the bounds listed in Table~\ref{Table:A}.
Finally, when $\ell=3$, the table is complete up to $\num{2e14}$.
\end{proposition}

%
%

Note that Propositions~\ref{P:0} and~\ref{P:1} imply the truth of Theorems~\ref{T:19} and~\ref{T:list}.
In the case of $\ell=3$, it is known that all $13$ of the fields listed in Table~\ref{Table:C}
are norm-Euclidean (see~\cite{smith:1969}).
In the case of $\ell=5$, Godwin~\cite{godwin:1965B} proved that $f=11$ is norm-Euclidean and
Cerri~\cite{cerri:2007} has verified this.  Up until this point,
it seems that nothing was known about the remaining fields in the table.
We use the algorithm of Cerri from~\cite{cerri:2007} (which has recently been extended by the first author in~\cite{lezowski})
with some additional modifications to show
that all five fields with $\ell=5,\,7$ in Table~\ref{Table:C} are norm-Euclidean.
In fact, we compute the Euclidean minimum
\[
  M(K)=\sup_{\xi\in K} m_K(\xi), \; \text{where }
m_K(\xi) = \inf_{\gamma\in\OK} |N(\xi-\gamma)|,
\]
for each of these fields.
It is well-known (and readily observed) that $M(K)<1$ implies that $K$ is norm-Euclidean.
\begin{proposition}\label{P:3}
Table~\ref{Table:E}
gives the Euclidean minimum $M(K)$ of the cyclic field $K$ having degree $\ell$ and conductor $f$.
\end{proposition}
\begin{table}
\begin{tabular}{LLLLLL}
\toprule
\ell & \multicolumn{3}{c}{5} & \multicolumn{2}{c}{7} \\  \cmidrule(lr){2-4} \cmidrule(lr){5-6} 
f & 11 & 31 & 41 & 29 & 43 \\
M(K) & 1/11 & 25/31 & 27/41 & 17/29 & 37/43 \\ \bottomrule
\end{tabular}
\caption{Euclidean minima}\label{Table:E}
\end{table}

It appears that the fields of degree $11$ are currently out of reach of the algorithm;
problems arise both from the time of computation required and
from issues related to precision.
In light of the discussion above, observe that the truth of the previous three propositions immediately
implies Theorems~\ref{T:quintic.septic},~\ref{T:19}, and~\ref{T:list}.
We detail the computations necessary to justify Propositions~\ref{P:1} and \ref{P:3}
in Sections~\ref{S:comp1} and~\ref{S:comp3} respectively.

In Section~\ref{S:bounds} we derive the conductor bounds given in Proposition~\ref{P:0}.
This involves a trick which allows us to weaken the condition for ``non-norm-Euclideanity''
from~\cite{mcgown:ijnt} provided $\ell>3$.
We are also able to accomplish this in the cubic case by a different argument
that takes advantage of the fact that the character takes only three values.
This is carried out in Section~\ref{S:cubic}.


Finally, the remainder of the paper 
is devoted to supplying the necessary justification for the upper bound on the conductor given in Theorem~\ref{T:cubic.new}.
The proof involves applying some recent results of Trevi\~no concerning non-residues
(see~\cite{trevino:mjcnt, trevino:jnt, trevino:ijnt}) together with ideas in~\cite{mcgown:jnt, mcgown:ijnt}.
For completeness, we provide improved (unconditional) conductor bounds for degrees $\ell>3$ as well.
However, as is the case when $\ell=3$,
these bounds are currently beyond our computational limits.
\begin{proposition}\label{P:2}
Table~\ref{Table:B} gives (unconditional)
conductor bounds for norm-Euclidean fields in $\mathcal{F}$.
\end{proposition}

\begin{table}
\centering
\begin{tabular}{lllllllll}
\toprule
$\ell$ & $3$ & $5$ & $7$ & $11$ & $13$ & $17$ & $19$ & $23$\\
\cmidrule{2-9}
$f<$ & $\num{e50}$ & $\num{e55}$ & $\num{e59}$ & $\num{e64}$ & $\num{e66}$ & $\num{e68}$ & $\num{e69}$ & $\num{e71}$  \\
\midrule
\end{tabular}
\begin{tabular}{lllllllll}
$\ell$ & $29$ & $31$ & $37$ & $41$ & $43$ & $47$ & $53$ & $59$  \\
\cmidrule{2-9}
$f<$ & $\num{e73}$ & $\num{e74}$ & $\num{e75}$ & $\num{e76}$ & $\num{e77}$ & $\num{e77}$ & $\num{e78}$ & $\num{e79}$ \\
\midrule
\end{tabular}
\begin{tabular}{lllllllll}
$\ell$ & $61$ & $67$ & $71$ & $73$ & $79$ & $83$ & $89$ & $97$  \\
\cmidrule{2-9}
$f<$ & $\num{e80}$ & $\num{e80}$ & $\num{e81}$ & $\num{e81}$ & $\num{e82}$& $\num{e82}$ & $\num{e83}$ & $\num{e84}$   \\
\bottomrule
\end{tabular}
\caption{Conductor bounds for NE fields in $\mathcal{F}$}\label{Table:B}
\end{table}



\section{Computation for Proposition~\ref{P:1}}\label{S:comp1}

Let $K$ denote the cyclic field of prime degree $\ell$ and conductor $f$.
We suppose that $K$ has class number one.
Assume that $(f,\ell)=1$ so that $K$ is not the field with $f=\ell^2$.
No field of this type having $f=\ell^2$ is norm-Euclidean anyhow except for
$\Q(\zeta_9+\zeta_9^{-1})$; this is~\cite[Thereom~4.1]{mcgown:ijnt}.
We may assume that $f$ is a prime with $f \equiv 1\pmod{\ell}$;
see~\cite[Section 2.1]{mcgown:ijnt}.
Denote by $q_1<q_ 2$ the two smallest rational primes that are inert in $K$.
Let $\chi$ denote any fixed primitive Dirichlet character of modulus $f$ and order $\ell$ so that
a rational prime $p$ splits in $K$ if and only if $\chi(p)=1$.
The following theorem is proved in~\cite{mcgown:ijnt}.
\begin{proposition}[{\cite[Theorem~3.1]{mcgown:ijnt}}]\label{P:conditions}
  Suppose that there exists $r\in\Z^+$
  satisfying the following conditions:
\begin{gather*}
    (r,q_1 q_2)=1,
    \quad
    \chi(r)=\chi(q_2)^{-1},\\
    r q_2 k\not\equiv f \pmod{q_1^2}
    \;
    \text{ for all }
    k=1,\dots,q_1-1, \\
    (q_1-1)(q_2 r-1)\leq f.
\end{gather*}
   Then $K$ is not norm-Euclidean.
\end{proposition}
Let $\mathcal{N}$ denote the image of the norm map from $\OK$ to $\Z$.
The proof of the previous proposition relies on:
\begin{lemma}[Heilbronn's Criterion]\label{L:heilbronn}
If one can write $f=a+b$, with $a,b>0$, $\chi(a)=1$, $a \notin \mathcal{N}$, 
$b\notin \mathcal{N}$, then
$K$ is not norm-Euclidean.
\end{lemma}
The advantage of Proposition~\ref{P:conditions} is that it requires far fewer steps
than applying Lemma~\ref{L:heilbronn} directly.
As $K$ has class number
one, we have that an integer $n \neq 0$ lies in $\mathcal{N}$
if and only if $\ell$ divides the $p$-adic valuation of $n$ for all
primes $p$ which are inert in $K$ (i.e. all primes $p$ for which
$\chi(p) \neq 0,1$).

To prove Proposition~\ref{P:1} we find $q_1,q_2,r$ as described above.
To save time, we look only for prime values of $r$.  See~\cite{mcgown:ijnt} for the details.
By applying Proposition~\ref{P:conditions}, we show there are no norm-Euclidean fields of the given form with $\num{e4}\leq f\leq F_\ell$ where
$F_\ell=\num{e13}$ when $3\leq \ell\leq 50$ and 
$F_\ell$ equals the value in Table~\ref{Table:A} when $50\leq\ell\leq 100$.
For example, Table~\ref{table:sample} shows the values for the last ten fields in our calculation when $\ell=97$.

To better manage the computation, the values of $f$ being considered (for a given~$\ell$) were broken into
subintervals of length $\num{e9}$.
As mentioned earlier, this computation took $3.862$ years of CPU time on a 96-core computer cluster.
Another computation of the same nature was performed to check that there are no norm-Euclidean
fields when $\ell=3$ with $\num{e13}\leq f\leq 2\cdot \num{e14}$, which took an additional $3.104$ years of CPU time on the same cluster.
Combining the computation just described with the results mentioned in Section~\ref{S:intro} proves Proposition~\ref{P:1}.

\begin{table}\label{table:sample}
\begin{tabular}{S[table-format=14]S[table-format=1]S[table-format=1]S[table-format=4]}
\toprule
{$f$} & {$q_1$} & {$q_2$} & {$r$}\\
\midrule
59999999974303 & 2 & 3 & 431\\
59999999975273 & 2 & 3 & 1933\\
59999999977213 & 2 & 3 & 241\\
59999999979929 & 2 & 3 & 673\\
59999999981869 & 2 & 3 & 797\\
59999999989823 & 2 & 3 & 2719\\
59999999990599 & 2 & 3 & 199\\
59999999995643 & 2 & 3 & 383\\
59999999999717 & 2 & 3 & 3709\\
59999999999911 & 2 & 3 & 2663\\
\bottomrule
\end{tabular}
\caption{Example calculation}
\end{table}




%
%
%


\section{Computation for Proposition~\ref{P:3}}\label{S:comp3}

Previously, Cerri computed that the Euclidean minimum
of the cyclic quintic field with conductor $f=11$ is equal to $1/11$
(see~\cite{cerri:2007}).
We compute the Euclidean minimum of the
remaining four fields $K$ with $\ell=5,\,7$ in Table~\ref{Table:C},
using the algorithm described in~\cite{lezowski}.

The algorithm is divided into two main parts:
\begin{itemize}
\item
Using an embedding of $K$ into $\R^\ell$,
we try to find a finite list of points $\mathcal{L} \subseteq K$
and some real number $k$ such that any point $x \in K \setminus \mathcal{L}$ we have $m_K(x) < k$.
If $k < 1$ and $\mathcal{L} = \emptyset$,
this proves that $K$ is norm-Euclidean.
\item
Next, we compute the Euclidean minimum of the remaining
points in $\mathcal{L}$. If $\max \{ m_K(x)\mid x \in \mathcal{L} \} > k$, then
\[
M(K) = \max \{ m_K(x)\mid x \in \mathcal{L} \}. 
\]
If not, we start again with smaller $k$.
\end{itemize}
The algorithm also returns the finite set of \emph{critical points}, that
is to say the points $x \in K/\OK$ satisfying $M(K) = m_K(x)$.
The results obtained are given in Table \ref{table:computation_euc_minima},
where $C$ is the cardinality of the set of critical points.

\begin{table}
\begin{tabular}{LLLLL}
\toprule
  \ell & f  & M(K) & C & \text{CPU time} \\ \midrule
5 & 31  & 25/31 & 10 & 13.2 \text{ min}\\
5 & 41  & 27/41 & 10 & 67.0 \text{ min}\\
7 & 29 &  17/29 & 14 & 95.6 \text{ min}\\
7 & 43 &  37/43 & 14 & 475.8 \text{ min}\\
\bottomrule
\end{tabular}
\caption{Computation of the Euclidean minima}\label{table:computation_euc_minima}
\end{table}

In carrying out these computations, the second part of the algorithm takes far longer than the first.
Nevertheless, we can improve
the running time with the following observation:
The points considered in our four cases are always of the form $\alpha/\beta$ where $\alpha$, $\beta$
are nonzero elements of $\OK$ such that $N\beta$ is the conductor $f$.
This provides some information on the Euclidean minimum of $\alpha/\beta$.

\begin{lemma}\label{lemma:Euclidean_minimum}
Let $\mathcal{N}$ denote the image of the norm map $N$.
Let $\alpha$ and $\beta$ be nonzero elements of $\OK$ such that
$N\beta = f$. Then
\[
  m_K \left( \frac{\alpha}{\beta} \right) \geq \frac{\min \left\{ |t| : t \in \mathcal{N},\, t \equiv N\alpha \pmod f \right\}}{f}.
\]
\end{lemma}
\begin{proof}
By definition of the Euclidean minimum,
\[
m_K \left( \frac{\alpha}{\beta} \right) = \frac{ \min \left\{ |N(\alpha-\beta z)| : z \in \OK \right\}}{N\beta} = \frac{\min \left\{ |N(\alpha-\beta z)| : z \in \OK \right\}}{f}.
\]
But for any $z\in \OK$, 
$N(\alpha-\beta z) \equiv N\alpha \pmod f$. As $N(\alpha-\beta z)$ is obviously an element
of $\mathcal{N}$, the result follows.
\end{proof}

In particular, for a point $\alpha/\beta$ of this form, 
if we can find some $z \in \OK$ such that
\[
|N(\alpha-\beta z)| =  \min \left\{ |t| : t \in \mathcal{N},\, t \equiv N\alpha \pmod f \right\},
\]
we will then have
\[
  m_K \left( \frac{\alpha}{\beta} \right) = \frac{|N(\alpha-\beta z)|}{f}
  .
\]


\bigskip

To illustrate this idea, consider the field $K=\Q(x)$
where $x^5 - x^4 - 12x^3 + 21x^2 + x - 5 =0$, of degree $5$ and conductor $31$.
At the end of running the algorithm, we get a list
$\mathcal{L}$ of ten points of the form
$\alpha/\beta$ as above.
One of the points found has
$\alpha=-\frac{106}{5}x^4 - \frac{162}{5}x^3 + \frac{866}{5}x^2 - \frac{28}{5}x - 41$, 
and $\beta = -4x^4 - 6x^3 + 33x^2 - 2x - 9$.
Then $N\alpha = -25$ (and $N\beta = 31$).
As $6 \notin \mathcal{N}$, Lemma \ref{lemma:Euclidean_minimum} implies that
 $m_K \left( \alpha / \beta\right) \geq 25/31$.  Of course, we have an equality because
$\left|N \left( \alpha/\beta\right) \right|=25/31$.
Besides, the ten points found are in the same orbit under the action of the units
on $K/\OK$. Thus their Euclidean minimum is equal to $25/31$, which
is the Euclidean minimum of $K$, and they are the set of critical
points, which has cardinality $10$.

\begin{remark}
Lemma~\ref{lemma:Euclidean_minimum} is a variation on Heilbronn's Criterion
(Lemma~\ref{L:heilbronn}). 
If 
\[
  \min \left\{ |t| : t \in \mathcal{N}, t \equiv N\alpha \pmod f \right\} > f,
\]
then we can deduce
from it an equality
\[
  f= a+b,
\]
where $a$ is the integer in $(0,f)$ such that $N\alpha \equiv a \pmod f$ and
$b = f-a$.
\end{remark}

%

\begin{remark}
The algorithm may also be applied to calculate the Euclidean minimum of cyclic cubic
number fields.  Table~\ref{table:cubic_algorithm} presents the results obtained in some of the
norm-Euclidean cases where the Euclidean minimum was previously unknown.
For conductors $f<103$, one can refer to \cite{lemmermeyer}.  As observed in \cite{godwin.smith:1993},
the field with conductor $157$ seems harder to deal with;
to date, no one has successfully computed the Euclidean minimum of this field.
\end{remark}

\begin{table}
\begin{tabular}{LLLLL}
\toprule
   f  & M(K)   & C & \text{CPU time} \\ \midrule
 103  & 93/103 & 6 & 12\text{ minutes}\\
 109  & 76/109 & 6 & 1\text{ day } 20\text{ hours}\\
 127  & 94/127 & 6 & 2\text{ days } 17\text{ hours} \\
\bottomrule
\end{tabular}
\caption{Euclidean minima in some cyclic cubic cases}\label{table:cubic_algorithm}
\end{table}

\bigskip

%


\section{Improved GRH conductor bounds when \texorpdfstring{$\ell>3$}{the degree is larger than 3}}\label{S:bounds}

We adopt the notation given in the first paragraph of Section~\ref{S:comp1}.
In addition, from now on and throughout the paper,
$r\in\Z^+$ will denote the smallest positive integer such that
$(r, q_1 q_2)=1$ and $\chi(r)=\chi(q_2)^{-1}$.
However, we do not assume any congruence conditions on $r$.
The following lemma is an improvement of statement (3) from
Theorem~3.1 of~\cite{mcgown:ijnt}; it is essentially a direct application of
the same theorem.

\begin{lemma}\label{lemma:other_bound}
Let us assume $q_1 \neq 2,3,7$. If $K$ is norm-Euclidean, then
\[
  f < \max \left\{q_1 ,\; \frac{2.1}{\ell} f^{1/\ell} \log f \right\} q_2 r.
\]
\end{lemma}

\begin{proof}
Let $u$ be the integer such that $0<u<q_1$ and $uq_2r \equiv f \pmod{q_1}$.
We set $v = (f-uq_2r)/q_1$, so that
$f = uq_2r+q_1v$.
This equation can be used with Heilbronn's Criterion
provided $v>0$ and $q_1 v\notin\mathcal{N}$.
Clearly $v\neq 0$ lest we contradict the fact that $f$ is prime.
Therefore, as we are assuming $K$ is norm-Euclidean,
it must be the case that $v<0$ or $q_1 v\in\mathcal{N}$.
However, $v<0$ immediately implies that $f<q_1q_2 r$,
and there is nothing more to prove.
Hence it suffices to assume $v>0$.  In this case we must
have $q_1 v\in\mathcal{N}$ which implies $q_1^{\ell-1}$ divides $v$.
Now we see that $v>0$ leads to $q_1^{\ell-1}\leq v<f/q_1$
and hence $q_1 < f^{1/\ell}$. 
As $q_1\neq 2,3,7$, we know from~\cite[Theorem~3.1]{mcgown:ijnt} that
$f<2.1 q_1 q_2 r \log q_1$ and
the result follows.
\end{proof}

\begin{proposition}\label{proposition:general_GRH}
Assume the GRH. If $K$ is norm-Euclidean and $f > \num{e9}$, then
\[
  f \leq \max \left\{ (1.17 \log(f)-6.36)^2, \;\frac{2.1}{\ell} f^{1/\ell} \log(f)  \right\} \cdot \left( 2.5(\ell-1) \log(f)^2 \right)^2.
\]
\end{proposition}
\begin{proof}
We use the bound on $q_1$ given in~\cite{bach} and the bounds on
$q_2$ and $r$ given in~\cite[Theorems~3.1 and~3.2]{mcgown:jtnb}.
If $q_1 \neq 2,3,7$, Lemma \ref{lemma:other_bound} together with these bounds
gives the result.  This completes the proof in most cases, but it remains to check that we obtain
better bounds in the other special cases.

If $q_1 = 7$, then $f < 21 \log 7 \cdot q_2 r$ by~\cite[Theorem~3.1]{mcgown:ijnt}.
We easily see that $(1.17 \log f -6.36)^2 > 21 \log 7$ for any $f > 60,000$.
The result now follows from~\cite[Theorem~3.2]{mcgown:jtnb}.
If $q_1=2,3$ and $q_2 > 5$, then we obtain $f < 5 q_2 r$
from~\cite[Theorem~3.1]{mcgown:ijnt} and the result follows in the same manner.

Finally, it remains to treat the two special cases: $(q_1,q_2) = (2,3), (3,5)$.
At this point, we assume $f\geq \num{e9}$.
Proposition~5.1 of~\cite{mcgown:ijnt} tells us that $f$ is bounded above by
$72(\ell-1)f^{1/2}\log(4f)+35$ and $507(\ell-1)f^{1/2}\log(9f)+448$
in the first and second case respectively.  In either case, the quantity in question
is bounded above by $568(\ell-1)f^{1/2}\log f$.  Consequently, we have $f\leq (568(\ell-1)\log f)^2$.
Now, one easily checks that  $(1.17\log f-6.36)^2(2.5)^2\geq 1442$ and $568^2(\ell-1)^2(\log f)^2\leq 1442(\ell-1)^2(\log f)^4$, which implies the desired result.
\end{proof}

Invoking the previous proposition immediately yields the GRH conductor bounds
given in Table~\ref{Table:A} when $\ell>3$.
The $\ell=3$ entry of Table~\ref{Table:A} will be obtained in Corollary~\ref{corollary:cubic_GRH}.
This completes the proof of
Theorems~\ref{T:quintic.septic},~\ref{T:19}, and~\ref{T:list}.



\section{The cyclic cubic case revisited}\label{S:cubic}

Unfortunately, the trick employed in the previous section does not help us when $\ell=3$.
Nonetheless, in the cubic case, we are able to slightly weaken the conditions for ``non-norm-Euclideanity''
given in~\cite[Theorem~3.1]{mcgown:ijnt}.  Notice that the following result contains no congruence conditions
and there is no extra $\log q_1$ factor.
The proof again relies on Heilbronn's Criterion (Lemma~\ref{L:heilbronn}),
but we will take advantage of the fact that $\chi$ only takes three different values in this very special case.

\begin{proposition}\label{proposition:bound_on_f}
Let $K$ be a cyclic cubic number field.
If $q_1 \neq 2$ and $f \geq q_1q_2\max(3r,10q_1)$, then $K$ is not norm-Euclidean.
\end{proposition}


\begin{proof}
It will be crucial that
$\ell=3$, which of course implies that
$\chi$ only takes three values:
$1$, $\chi(q_2)$, and $\chi(r)$. In addition, we have
$\chi(r)= \chi(q_2)^{-1}=\chi(q_2)^2$.

Let $u$ be the integer such that $0<u<q_1$ and $uq_2r \equiv f \pmod {q_1}$.
We set $v = (f-uq_2r)/q_1$, so that
\begin{equation}\label{eq:bound_on_f:1}
f = uq_2r+q_1v.
\end{equation}
Observe that $f\geq q_1 q_2 r$ implies $v\geq 0$; moreover,
we may assume $v\neq 0$ lest we contradict the fact that
$f$ is a prime.
If $q_1$ does not divide $v$, then we may apply Heilbronn's Criterion with 
\eqref{eq:bound_on_f:1}.
Hence we may assume that $q_1$ divides $v$.
We break the proof into a number of cases.

\begin{enumerate}
\item
Suppose $u$ is odd.  Then $u+q_1$ is even and $0 < (u+q_1)/2 < q_1$,
so every prime divisor $p$ of $(u+q_1)/2$ is such that $\chi(p)=1$.
Besides, $(q_1,q_2r)=1$ and $q_1$ divides $v$, so $q_1$ does not
divide $v-q_2r$.
As $q_1 > 2$ and $\chi(2) = 1$, we may apply Heilbronn's Criterion
with 
\begin{equation}\label{eq:bound_on_f:2}
\phantom{a}
\end{equation}

\vspace{-3em}

\noindent\begin{minipage}[t]{\linewidth}
\[
    f = (u+q_1)q_2r + q_1(v-q_2r),
\]
provided $v \geqslant q_2r$.
\end{minipage}
\item
Suppose $u$ is even.  We distinguish cases according to the value of $u+q_1$.
  \begin{enumerate}
  \item
  If $u+q_1$ is composite, then any prime factor $p$ of $u+q_1$
    is such that $p<q_1$. Therefore, we may again apply Heilbronn's Criterion
    with \eqref{eq:bound_on_f:2}, provided $v \geqslant q_2r$.
  \item
  If $u+q_1$ is prime and $\chi(u+q_1)=1$, then we proceed similarly.
  \item
  Suppose $u+q_1$ is prime and $\chi(u+q_1) = \chi(r)$.
  Notice that
  $r \leq u+q_1$.
  \begin{enumerate}
    \item If $\frac{u}{2} + q_1$ is composite or a prime such 
    $\chi(\frac{u}{2}+q_1) = 1$, then 
    $(q_2,\frac{u}{2}+q_1)=1$ and we may use 
    Heilbronn's Criterion with
    \begin{equation}\label{eq:bound_on_f:4}
       \phantom{a}
    \end{equation}
    
    \vspace{-3em}

    \begin{minipage}[t]{\linewidth}
    \[
        f = (u+2q_1)q_2r + q_1(v-2q_2r),
    \]
    provided $v \geq 2q_2r$.
    \end{minipage}
    \item \label{item:subcase_c_ii}
    If $\frac{u}{2}+q_1$ is prime and $\chi(\frac{u}{2}+q_1) = \chi(q_2)$,
    then we have $q_2 \leq \frac{u}{2}+q_1 < \frac{3}{2}q_1$
    which also implies
    $0 \leq u + 2(q_1-q_2) < q_1$.
    If $u \neq 2(q_2-q_1)$, then 
    $q_1$ does not divide $v+r(u+2(q_1-q_2))$ and
    therefore we can apply Heilbronn's Criterion
    with \\
    \noindent\begin{minipage}[t]{\linewidth}
    \vspace*{-1em}
    \[
       f = (u+2q_1)(q_2-q_1)r + q_1(v+r(u+2(q_1-q_2))),
    \]
    provided $v \geq 2q_2r$. Indeed, if $(u+2q_1)(q_2-q_1)r \in \mathcal{N}$,
    then the valuation of $(u+2q_1)(q_2-q_1)r$ at $\frac{u}{2}+q_1$ is at least
    $\ell =3$, so
    $(\frac{u}{2}+q_1)^2$ divides $r \leq u+q_1$.
    Then $q_1^2 < (\frac{u}{2}+q_1)^2 \leq r < 2q_1$, which is impossible.

    If $u = 2(q_2-q_1)$, then $4$ divides $u$ and
    $q_1 < \frac{u}{4} + q_1 < q_2$. Therefore, $(q_2,(u+4q_1)r)=1$ and
    $(u+4q_1)q_2r \notin \mathcal{N}$.
    So we may apply Heilbronn's Criterion
    with \\
    \noindent\begin{minipage}[t]{\linewidth}
    \vspace*{-1em}
    \[
      f= (u+4q_1)q_2r + q_1(v-4q_2r),
    \]
    provided $v \geq 4q_2r$. Notice in this case
    that $q_2 \leq \frac{3}{2} q_1$ and
    $r \leq u+q_1 = 2q_2-q_1 \leq 2q_1$.
    \end{minipage}
    \end{minipage}
    \item
    If $\frac{u}{2}+q_1$ is prime and $\chi(\frac{u}{2}+q_1) = \chi(r)$,
    then $q_2 < r \leq \frac{u}{2} + q_1$. Therefore, $(q_2,u+2q_1)=1$.
    Besides, $r-q_1 < q_1$, so $(q_2,r-q_1)=1$. As a result,
    $(u+2q_1)q_2 (r-q_1) \notin \mathcal{N}$.
    If $r \neq \frac{u}{2}+q_1$, then $q_1$ does not divide
    $q_2(u+2q_1-2r)$, so we can apply Heilbronn's Criterion with \\
    \noindent\begin{minipage}[t]{\linewidth}
    \vspace*{-1em}
    \[
      f= (r-q_1)q_2(u+2q_1) + q_1(v+q_2(u+2q_1-2r)),
    \]
    assuming $v \geq 2q_2r$.

    If $r=\frac{u}{2} + q_1$, then $u+q_1-r = \frac{u}{2}$.
    As $u+q_1$ is prime and satisfies
    $\chi(u+q_1) =\chi(r)$,
    we have $(u+q_1,(r-q_1)q_2)=1$ and
    $(r-q_1)q_2(u+q_1) \notin \mathcal{N}$.
    So we may use
    Heilbronn's Criterion with \\
    \noindent\begin{minipage}[t]{\linewidth}
    \vspace*{-1em}
    \[
       f = (r-q_1)q_2(u+q_1) + q_1(v+q_2 (u+q_1-r)),
    \]
    assuming $v \geq 0$.
    \end{minipage}
    \end{minipage}
  \end{enumerate}
  \item
  Suppose $u+q_1$ is prime and $\chi(u+q_1) = \chi(q_2)$.
  
  \begin{enumerate}
      \item If $\frac{u}{2}+q_1$ is composite or a prime such
        that $\chi(\frac{u}{2}+q_1)=1$, then we may use
        Heilbronn's Criterion with \eqref{eq:bound_on_f:4}.
      \item If $\frac{u}{2}+q_1$ is prime and
        $\chi(\frac{u}{2}+q_1)=\chi(q_2)$, then
        $q_2 \leq \frac{u}{2}+q_1 < u+q_1$, so
        $(q_1q_2,u+q_1)=1$, and by definition of
        $r$, $r \leq (u+q_1)^2$.
        
        If $r<(u+q_1)^2$, then $(u+q_1)(q_2-q_1)r \notin
        \mathcal{N}$. Indeed, $q_2 \leq \frac{u}{2}+q_1$,
        so $q_2 - q_1 < q_1$. Besides, $(u+q_1)^2$
        cannot divide $r<(u+q_1)^2$. Consequently,
        we may apply Heilbronn's Criterion with
        \begin{equation}\label{eq:bound_on_f:5}
            \phantom{a}
        \end{equation}
    
        \vspace{-3em}

        \noindent\begin{minipage}[t]{\linewidth}
        \[
            f = (u+q_1)(q_2-q_1)r+q_1(v+r(u+q_1-q_2)),
        \]
        assuming $v \geq q_2 r$. Indeed,
        $q_2 \leq u+q_1 < 2q_1$, so $0  \leq u + q_1-q_2
        < q_1$, and $u=q_2 -q_1$ is impossible, 
        because it would imply $q_1 < \frac{u}{2} + q_1 < q_2$,
        which contradicts $\chi(\frac{u}{2}+q_1) \neq 1$.
   
           If $r=(u+q_1)^2$, then $(\frac{u}{2}+q_1)^2 < r$;
           in this case, it follows from the definition of $r$ that
          $(\frac{u}{2}+q_1,q_2) \neq 1$ and we obtain
          $\frac{u}{2} + q_1 = q_2$. We may now apply     
        Heilbronn's Criterion with \\
        \noindent\begin{minipage}[t]{\linewidth}
        \vspace*{-1em}
        \[
          f=2uq_2^2(u+q_1)+q_1(v-uq_2(u+q_1)),
        \]
        assuming $v \geq u q_2(u+q_1)$ (which
        holds if $v \geq q_2(u+q_1)^2 = q_2 r$).
        \end{minipage}
        \end{minipage}
      \item If $\frac{u}{2}+q_1$ is prime
      and $\chi(\frac{u}{2}+q_1)=\chi(r)$, then 
      $q_2 < r \leq \frac{u}{2}+q_1$.
      Therefore, $(u+q_1,(q_2-q_1)r)=1$ and
      $(u+q_1)(q_2-q_1)r \notin \mathcal{N}$. 
      Besides, $0 < q_2-q_1 < u+q_1 -q_2 < 2q_1 -q_2 < q_1$,
      and we may
      apply Heilbronn's Criterion with \eqref{eq:bound_on_f:5},
      assuming $v\geq 0$.
  \end{enumerate}
  \end{enumerate}
\end{enumerate}
Now we summarize.  In all cases but one, the assumption $v\geq 2 q_2 r$ is sufficient
and hence it is enough to require that $f\geq 3 q_1 q_2r$.
(Recall that $v=(f- u q_2 r)/q_1$.)
In the exceptional case, we have shown that
$v\geq 4 q_2 r$ is sufficient; but in that situation we also know $r\leq 2q_1$ and
therefore it is enough to require that $f\geq 10 q_1^2 q_2$.
\end{proof}

\begin{corollary}\label{corollary:cubic_GRH}
Assume the GRH.  Let $K$ be a cyclic cubic field.
If $K$ is norm-Euclidean, then $f < \num{4e10}$.
\end{corollary}
\begin{proof}
We use Proposition \ref{proposition:bound_on_f} and the bounds on $q_1$, $q_2$ 
and $r$ given in \cite{bach, mcgown:jtnb}.
\end{proof}

Although the previous corollary is already known, we want to point out that Proposition~\ref{proposition:bound_on_f} allows one to prove
Theorem~\ref{T:cubic.GRH} using less computation than is employed in~\cite{mcgown:jtnb}.
More importantly, Proposition~\ref{proposition:bound_on_f} will serve as one of the main ingredients in lowering the unconditional conductor bound
(in the cubic case).


\section{Character non-residues}\label{S:nonresidues}

Let $\chi$ be a non-principal Dirichlet character modulo a prime $p$.
Suppose that $q_1<q_2$ are the two smallest prime non-residues of $\chi$.
This section is devoted to improving the constants appearing in~\cite{mcgown:jnt}.
We begin by quoting a result proved by Trevi\~no:

\begin{proposition}[{\cite[Theorem~1.2]{trevino:jnt}}]\label{P:trevino}
Suppose $p>3$.  Then $q_1<0.9\,p^{1/4}\log p$
unless $\chi$ is quadratic and $p\equiv 3\pmod{4}$,
in which case
$q_1<1.1\,p^{1/4}\log p$.
\end{proposition}

The following proposition will lead to improved bounds on $q_2$ and the product
$q_1 q_2$:

\begin{proposition}\label{P:q_2}
Suppose $p\geq \num{e6}$, and that $u$ is a prime with $u\geq A\log p$
where
$A=(2/5)e^{3/2}\approx 1.79$.
Suppose
$\chi(n)=1$ for all $n\in[1,H]$ with $(n,u)=1$.  Then
\[
  H\leq g(p)\,p^{1/4}\log p
  \,,
\]
where $g(p)$ is an explicitly given function.  Moreover,
$g(p)$ is decreasing for $p\geq 10^6$ and $g(p)\to 2.71512...$.
\end{proposition}

\begin{proof}
Similar to the proof of~\cite[Theorem~3]{mcgown:jnt}, we may reduce to the case where
$H\leq (A\log p-1)^{1/2}p^{1/2}$.
We may assume $H\geq Kp^{1/4}\log p$ where $K:=2.7151$, or else there is nothing to prove.
We set $h=\lfloor A\log p\rfloor$, $r=\lceil B\log p\rceil$,
with $A=(2/5)e^{3/2}$ and $B=1/5$.

For $1/2 \leq y \leq 1$, we have $\exp(y/2B) \leq \exp(1/2B) y$, so in particular
\begin{equation*}
p^{1/2r} = \exp \left( \frac1{2B} \frac{B \log(p)}{r}\right)
\leq \exp \left( \frac{1}{2B} \right)  \frac{B \log(p)}r.
\end{equation*}
But $B\exp(1/2B) \log(p) = eA/2 \cdot \log(p) \leq e(h+1)/2$, from which we deduce
\begin{equation}\label{E:chooseAB}
  \left( \frac{2r}{e(h+1)}\right)^r \leq p^{-1/2}.
\end{equation}

One verifies that $Kp^{1/4}> 32 A$ for $p\geq \num{e6}$ and hence $H> 32h$.
We set $X:=H/(2h)$ and observe that we have the a priori lower bound
\[
  X=
  \frac{H}{2h}
  \geq
  \frac{Kp^{1/4}\log p}{2A\log p}
  =
  \frac{Kp^{1/4}}{2A}
  \,,
\]
and, in particular, $X>16$ from the previous sentence.
We will make use of the function $f(X,u)$ defined by
  \[
  f(X,u)=1-\frac{\pi^2}{3}
  \left(
  \frac{1}{2X^2}+\frac{1}{2X}+\frac{1}{1-u^{-1}}\cdot\frac{1+\log X}{X}
  \right)
  \,;
\]
observe that
\[
 \hat{f}(p):=f\left(\frac{Kp^{1/4}}{2A},\,A\log p\right)\leq f(X,u)
  \,.
\]
Combining~\cite[Proposition~1]{mcgown:jnt}, \cite[Theorem~1.1]{trevino:jnt}, and
an explicit version of Stirling's Formula (see~\cite{spira:1971}, for example), we obtain
\begin{equation}
  \label{E:inequality1}
  \frac{6}{\pi^2}\left(1-u^{-1}\right)h(h-2)^{2r}\left(\frac{H}{2h}\right)^2\hat{f}(p)
  \;\leq\;
  \sqrt{2}\left(\frac{2r}{e}\right)^r p h^r + (2r-1)p^{1/2}h^{2r}.
\end{equation}
Using the convexity of the logarithm, we establish
the following estimates:
\begin{equation}
  \label{E:estimates}
  \left(\frac{h}{h-2}\right)^r\leq F(p)
  \,,\quad
  \left(\frac{h+1}{h}\right)^r\leq G(p)
\end{equation}
\[
  F(p):=\exp\left(\frac{2B\log p +2}{A\log p -3}\right)
  \,,\quad
  G(p):=\exp\left(\frac{B\log p+1}{A\log p-1}\right)
\]
Note that $F(p)$ and $G(p)$ are both decreasing functions of $p$.
Rearranging (\ref{E:inequality1}) and applying (\ref{E:estimates}), (\ref{E:chooseAB}) gives:
\begin{align}
&
\label{E:explain1}
\frac{6}{\pi^2}\left(1-u^{-1}\right)H^2\hat{f}(p)
\\
\notag
& \qquad
\leq
4h(2r-1)\left(\frac{h}{h-2}\right)^{2r}p^{1/2}
\left[
1+\frac{\sqrt{2}}{2r-1}\left(\frac{2r}{eh}\right)^rp^{1/2}
\right]
\\
\notag
& \qquad
\leq
4h(2r-1)F(p)^2
p^{1/2}
\left[
1+
\frac{\sqrt{2} G(p)}{2r-1}\left(\frac{2r}{e(h+1)}\right)^r p^{1/2}
\right]
\\
\notag
& \qquad
\leq
4(A\log p)(2B\log p+1)F(p)^2
p^{1/2}
\left(
1+
\frac{\sqrt{2}G(p)}{2r-1}
\right)
\\
\notag
&\qquad
\leq
8 ABp^{1/2}(\log p)^2 F(p)^2 
\left(
1+\frac{1}{2B\log p}
\right)
\left(
1+
\frac{\sqrt{2}G(p)}{2B\log p-1}
\right)
\\
&\qquad
\leq 
8AB\,p^{1/2}(\log p)^2
F(p)^2
\left(
1+\frac{5}{2\log p}
\right)
\left(
1+
\frac{5\sqrt{2}G(p)}{2\log p-5}
\right)
\end{align}
Now the result follows provided we define:
\[
g(p):=
2\pi
F(p)
\sqrt{
\frac{
A
\left(
1+\frac{5}{2\log p}
\right)
\left(
1+
\frac{5\sqrt{2}G(p)}{2\log p-5}
\right)
}
{
15
\hat{f}(p)
\left(
1-\frac{1}{A\log p}
\right)
}
}
\,.\hspace{3pt}\qedhere
\]
\end{proof}

\begin{proposition}\label{P:A1}
  Fix a real constant $p_0\geq \num{e6}$.
  There exists an explicit constant $C$ (see Table~\ref{Table:A1})
  such that if $p\geq p_0$ and $u$ is a prime with $u\geq 1.8\log p$,
  then there exists $n\in\Z^+$ with $(n,u)=1$, $\chi(n)\neq 1$, and
  $n<C\,p^{1/4}\log p$.

\begin{table}[t]
\centering
\begin{tabular}{lllllllll}
\toprule
$p_0$ & $\num{e6}$ & $\num{e8}$ & $\num{e10}$ & $\num{e12}$ & $\num{e14}$ & $\num{e16}$& $\num{e18}$& $\num{e20}$\\ \cmidrule{2-9}
$C$ & $6.9236$ & $4.1883$ & $3.5764$ & $3.3290$ & $3.2019$ & $3.1246$ & $3.0716$ & $3.0320$ \\
\midrule
$p_0$ & $\num{e22}$ & $\num{e24}$ & $\num{e26}$ & $\num{e28}$ & $\num{e30}$ & $\num{e32}$& $\num{e34}$& $\num{e36}$\\ \cmidrule{2-9}
$C$ & $3.0008$ & $2.9754$ & $2.9542$ & $2.9363$ & $2.9208$ & $2.9074$ & $2.8956$ & $2.8852$\\
\midrule
$p_0$ & $\num{e38}$ & $\num{e40}$ & $\num{e42}$ & $\num{e44}$ & $\num{e46}$ & $\num{e48}$& $\num{e50}$& $\num{e52}$\\ \cmidrule{2-9}
$C$ &  $2.8759$ & $2.8676$ & $2.8601$ & $2.8533$ & $2.8471$ & $2.8415$ & $2.8363$ & $2.8315$\\
\bottomrule
\end{tabular}
\caption{Values of $C$ for various choices of $p_0$\label{Table:A1}}
\end{table}
\end{proposition}

\begin{proof}
Let $n_0$ denote the smallest $n\in\Z^+$ such that $(n,u)=1$ and $\chi(n)\neq 1$.
We apply Proposition~\ref{P:q_2} to find
\[
  n_0-1\leq g(p_0)\,p^{1/4}\log p
  \,.
\]
Computation of the table of constants is routine;
for each value of $p_0$, we compute (being careful to round up)
the quantity
\[
g(p_0)+\frac{1}{p_0^{1/4}\log p_0}
\,.\hspace{3pt}\qedhere
\]
\end{proof}

\begin{corollary}\label{C:A2}
If $p\geq 10^{13}$,
then
$q_2<2.8\,p^{1/4}(\log p)^2$.
\end{corollary}

\begin{proof}
If $q_1>1.8\log p$, then we apply the previous proposition and we are done.
Hence we may assume that $q_1\leq 1.8\log p$.  If $q_2=3$, then
we are clearly done, so we may also assume $q_2\neq 3$.  In this case,
we combine \cite[Lemma~7]{mcgown:jnt}
and~\cite[Theorem~1]{trevino:mjcnt}\footnote
{
In a private correspondence, the author of~\cite{trevino:mjcnt} has indicated
that the result contained therein holds when $p>10^{13}$;
a correction to~\cite{trevino:mjcnt} is forthcoming.
}
to conclude $q_2\leq (1.55 p^{1/4}\log p)(1.8\log p)+1<2.8p^{1/4}(\log p)^2$.
\end{proof}

\begin{corollary}\label{C:A3}
Suppose $p\geq \num{e30}$ and that $\chi$ has odd order.
Then
\[
  q_1 q_2<2.64\,p^{1/2}(\log p)^2
  \,.
\]
\end{corollary}

\begin{proof}
If $q_1<1.8\log p$, we use the previous corollary (and its proof)
to obtain $q_2<3 p^{1/4}(\log p)^2$,
and hence $q_1 q_2<5.4 p^{1/4}(\log p)^3\leq 0.01 p^{1/2}(\log p)^2$.
If $q_1\geq 1.8\log p$, then we apply Proposition~\ref{P:trevino} (using the fact that $\chi$ has odd order)
and Proposition~\ref{P:q_2} to find
$q_1 q_2\leq (0.9 p^{1/4}\log p)(2.93 p^{1/4}\log p)$.
The result follows.
\end{proof}


\section{Improved unconditional conductor bounds}\label{S:newbounds}

In this section we will prove Proposition~\ref{P:2}.
First, we observe that applying Trevi\~no's version of the Burgess Inequality (see~\cite{trevino:ijnt})
immediately\footnote{In the technical condition appearing in the proposition,
one must replace $4 f^{1/2}$ by $2 f^{1/2}$;
however, in our application, this condition will be automatically met so
this has no real effect.  Moreover, it is shown in~\cite{trevino:ijnt}
that the technical condition may be dropped completely provided $k\geq 3$.}
gives better constants $D(k)$ for~\cite[Proposition~5.7]{mcgown:ijnt}
for $2\leq k\leq 10$.

\begin{table}
\begin{tabular}{LLL}
\toprule
  & \multicolumn{1}{c}{$q_1$ arbitrary} & \multicolumn{1}{c}{$q_1>100$}   \\ \cmidrule(lr){2-2} \cmidrule(lr){3-3}
k & D_1(k) & D_2(k)\\ \midrule
2 & 36.9582 & 5.6360\\
3 & 25.3026 & 3.8981\\
4 & 21.3893 & 3.3703\\
5 &  19.4132& 3.1104\\
6 & 18.2048 & 2.9523\\
7 & 17.3797 & 2.8439\\
8 & 16.7819 & 2.7650\\
9 & 16.3162 & 2.7030\\
10 & 15.9414 & 2.6525\\
\bottomrule
\end{tabular}
\caption{Values of $D(k)$ when $2\leq k\leq 10$\label{Tab:D} and $f\geq 10^{20}$}
\end{table}

To establish our result, we follow the proof of~\cite[Theorem~5.8]{mcgown:ijnt}.
Details that are identical or very similar will be omitted.
We may assume throughout that $f\geq \num{e40}$.
If $q_1\leq 100$, the techniques
in~\cite{mcgown:ijnt} already
give the desired result
and hence we may assume $q_1>100$.
We treat the cases of $\ell=3$ and $\ell>3$ separately.

First, we treat the cubic case.
In light of Proposition~\ref{proposition:bound_on_f},
it suffices to verify that
$10 q_1^2 q_2\leq f$ and $3 q_1 q_2 r\leq f$.
The former condition easily holds, since applying Proposition~\ref{P:trevino} and Corollary~\ref{C:A3}
immediately gives $10 q_1^2 q_2<24 f^{3/4}(\log f)^3<f$.
We turn to the latter condition.
Proposition~5.7 of~\cite{mcgown:ijnt} (with the improved constants) gives:
\[
  r\leq (D_2(k)(\ell-1))^k f^{\frac{k+1}{4k}}(\log f)^{\frac{1}{2}}
  \,.
\]
Applying Corollary~\ref{C:A3}, this leads to:
\begin{align*}
  3q_1 q_2 r
  &\leq
    3\cdot 2.64 f^{1/2}(\log f)^2\cdot (D_2(k)(\ell-1))^k f^{\frac{k+1}{4k}}(\log f)^{\frac{1}{2}}
  \\
  &\leq
  8\,D_2(k)^k (\ell-1)^k f^{\frac{3k+1}{4k}}(\log f)^\frac{5}{2},
\end{align*}
Choosing $k=4$ we see that the desired condition holds when $f\geq \num{e50}$.

Now we turn to the case when $\ell>3$.
In light of Lemma~\ref{lemma:other_bound},
it suffices to verify that
\[
   \max
   \left\{
   q_1
   \,,
   f^{1/5}\log f
   \right\}
   \,
   q_2
   r
   \leq f
\]
Using Corollaries~\ref{C:A2}, \ref{C:A3} we have
$q_1 q_2\leq 2.7 f^{1/2}(\log f)^2$ as well as
\[
  (f^{1/5}\log f) q_2\leq 3 f^{9/20}(\log f)^2<2.7 f^{1/2}(\log f)^2
  \,.
\]
Consequently, applying Proposition~5.7 of~\cite{mcgown:ijnt} as before, we now have:
\begin{equation}
  \label{E:final.condition}
   \max
   \left\{
   q_1
   \,,
   f^{1/5}\log f
   \right\}
   \,
   q_2
   r
  \leq
  2.7\,D_2(k)^k (\ell-1)^k f^{\frac{3k+1}{4k}}(\log f)^\frac{5}{2}.
\end{equation}
For the primes $\ell=5,7$ we use $k=4$ and for the remaining values of $\ell$
we use $k=3$.
We check that the expression on the righthand side of (\ref{E:final.condition})
is less than $f$ provided $f$ is greater than the value given in Table~\ref{Table:B}.

\section*{Acknowledgements}
The authors would like to thank Thomas Carroll of Ursinus College for graciously allowing the use of his computer cluster for this project.
The research of the first author was funded by a grant given by r\'egion Auvergne.
The second author was partially supported by an internal research grant from California State University, Chico. 
Both authors would like to thank the anonymous referee, whose remarks helped to improve the final version of the paper.

\end{document}